\documentclass[a4paper,12pt]{scrartcl}

\usepackage[utf8]{inputenc}
\usepackage{amsfonts,amssymb}
\usepackage{amsmath}
\usepackage{graphicx}

\usepackage[PostScript=dvips]{diagrams}  

\usepackage{color}

\usepackage{a4wide}

\usepackage{empheq}
\numberwithin{equation}{section}

\usepackage{amsthm}
\theoremstyle{plain}
\newtheorem{thm}{Theorem}[section]
\newtheorem*{mainthm}{Main Theorem}
\newtheorem{lem}[thm]{Lemma}

\newtheorem{prop}[thm]{Proposition}

\newtheorem{introbs}{Observation}

\theoremstyle{definition}

\theoremstyle{remark}

\newcommand{\N}{\mathbb{N}}
\newcommand{\Z}{\mathbb{Z}}

\usepackage{ifthen}
\usepackage{xargs}
\newcommand{\optionalsubsuper}[3]{%
\ifthenelse{\equal{#2}{}}{%
  \ifthenelse{\equal{#3}{}}{%
    #1}{%
    #1^{#3}}
  }{%
  \ifthenelse{\equal{#3}{}}{%
    #1_{#2}}{%
    #1_{#2}^{#3}}}}

\newcommand{\Union}{\bigcup}
\newcommand{\intersect}{\cap}

\newcommand{\defeq}{\mathrel{\mathop{:}}=}

\newcommand{\Res}[2]{\operatorname{Res}^{#1}_{#2}}
\newcommand{\Ind}[2]{\operatorname{Ind}_{#1}^{#2}}
\newcommand{\Family}{\mathfrak{F}}
\newcommandx{\OrbitCat}[2][1=\Group,2=\Family]%
   {\optionalsubsuper{\mathcal{O}}{#2}{}#1}
\newcommand{\BredonZ}{\underline{\Z}}
\newcommand{\dash}{-}
\newcommand{\PlaceA}{\dash}

\newcommand{\FP}[1]{\ensuremath{\mathrm{FP}_{#1}}}
\newcommand{\Tor}{\operatorname{Tor}}

\newcommand{\BredonHoml}[1]{H_{#1}^\Family}
\newcommand{\Homl}[1]{\underline{H\!}_{\,#1}}
\newcommand{\RHoml}[1]{\underline{\tilde{H}\!}_{\,#1}}
\newcommand{\Skeleton}[2]{#1^{(#2)}}

\newcommand{\isom}{\cong}
\newcommand{\Group}{\Gamma}
\newcommand{\Subgroup}{\Lambda}
\newcommand{\AnotherSubgroup}{\Xi}
\newcommand{\VC}{\mathcal{VC}}
\newcommand{\FIN}{\mathcal{FIN}}

\newcommand{\Ab}{\mathfrak{Ab}}
\renewcommand{\:}{\colon}
\DeclareMathOperator{\BredonRCat}{Mod-\OrbitCat[\Group][\Family]}
\DeclareMathOperator{\BredonLCat}{\OrbitCat-Mod}
\newcommand{\BredonC}[1]{\underline{C\!}_{\,#1}}

\newcommand{\epi}{\twoheadrightarrow}

\newcommandx{\BredonFP}[2][1=\Family]{\ensuremath{#1\textrm{-}\mathrm{FP}_{#2}}}
\DeclareMathOperator{\pt}{pt.}

\DeclareMathOperator{\mor}{mor}

\title{Brown's criterion in Bredon homology}

\author{Martin Fluch \and Stefan Witzel}

\begin{document}

\maketitle

\begin{abstract}
We translate Brown's criterion for homological finiteness properties to the setting of Bredon homology.
\end{abstract}

\thispagestyle{empty}


Bredon cohomology has become an important algebraic tool for studying
classifying spaces~$E_{\Family}\Group$ of discrete groups $\Group$
with stabilisers in a given family $\Family$ of subgroups of~$\Group$.
It has been defined for finite groups by Bredon~\cite{bredon-67} and
the definition has been extended to arbitrary groups and families of
subgroups by L\"{u}ck~\cite{luck-89}.
The basic idea in passing from classical cohomology to Bredon cohomology
is to replace $\Group$, regarded as a small category, by the orbit category
$\OrbitCat$.

More precisely, let $\Group$ be a discrete group.  By a \emph{family}
of subgroups of $\Group$ we mean a non-empty set $\Family$ of
subgroups of $\Group$ which is closed under conjugation.  The
transitive $\Group$-sets $\Group/\Subgroup$ with $\Subgroup\in \Family$
and $\Group$-maps between them form the \emph{orbit category} $\OrbitCat$.
A \emph{right (left) Bredon module over $\OrbitCat$} is a
contravariant (covariant) functor from $\OrbitCat$ to the category
$\Ab$ of abelian groups.  A morphism of Bredon modules of the same
variance is a natural transformation.  Right (left) Bredon modules and
the morphisms between them form a category which is denoted by
$\BredonRCat$ ($\BredonLCat$).

$\BredonRCat$ and
$\BredonLCat$ are functor categories and thefore they inherit many
properties from the category~$\Ab$.  Among others they are abelian
categories in which all small
limits and colimits exist, they have enough projectives and there
exists  a notion of
being finitely generated.  Details are given in
Section~\ref{sec:Bredon-basics} below. 

Let $n\in \N\cup \{\infty\}$.  An $\OrbitCat$-module $M$ is said to be
of \emph{type $\BredonFP{n}$} if there exists a resolution
\begin{equation*}
    \ldots\to P_{2}\to P_{1}\to P_{0}\to M \to 0
\end{equation*}
of $M$ by projective $\OrbitCat$-modules such that $P_{k}$ is finitely
generated for every $k\leq n$.  The trivial $\OrbitCat$-module
$\BredonZ$ maps every object of $\OrbitCat$ to $\Z$ and every morphism
of $\OrbitCat$ to the identity.  A group $\Group$ is said to be of
\emph{type $\BredonFP{n}$} if $\BredonZ$ is of type $\FP{n}$ as a
right $\OrbitCat$-module.

In the special case that $\Family = \{1\}$ consists only of the
trivial group Bredon cohomology reduces to classical
cohomology of groups.  The finiteness properties $\FP{n}$ in this case have been
extensively studied. The classical proofs sometimes also carry over to the
case where $\Family=\FIN$ is the family of finite subgroups. This is true for example of hyperbolic groups, arithmetic groups, mapping class groups, and outer automorphism groups of finitely generated free groups, see \cite[Sections~4.7,4.8]{luck-05} (and trivially for torsion free groups). On the other hand Leary and Nucinkis \cite{leary-03}
showed how much finiteness properties with respect to $\FIN$ can differ from the classical ones.

The next family of interest is the family $\VC$ of virtually cyclic subgroups.
One result here is that for an elementary amenable group
being of type \BredonFP[\VC]{\infty} is
equivalent being virtually cyclic~\cite{kochloukova-11}.

In the classical setting Brown's Criterion \cite[Theorem~2.2]{brown-87} has
been fruitful in the study of the properties $\FP{n}$. Our main result is a
translation of this criterion to the Bredon setting. In order to state 
it, some more definitions are needed.

\medskip

A \emph{$\Group$-CW-complex} $X$ is a CW-complex on which $\Group$
acts by cell-permuting homeomorphisms such that the stabilizer of a
cell fixes that cell pointwise. We let $\BredonC{*}(X)$ denote the Bredon
cellular chain complex of $X$, cf.~\cite[p.~11]{mislin-03}. The
\emph{Bredon homology modules} $\Homl{*}(X)$ of $X$ are defined to be
the homology modules of the Bredon chain complex $\BredonC{*}(X)$.
Evaluated at $\Group/\Subgroup\in 
\OrbitCat$ these Bredon modules give
\begin{equation*}
    \Homl{*}(X)(\Group/\Subgroup) = H_{*}(X^{\Subgroup})
\end{equation*}
where the right hand side is the ordinary homology of the fixed point
complex $X^{\Subgroup}$.  This definition is functorial.  Analogously
to the classical case we define the \emph{reduced Bredon homology
modules} $\RHoml{*}(X)$ to be the kernel of the morphism
$\Homl{*}(X)\to \Homl{*}(\pt)$ which is induced by the map from $X$ to
the singleton space.  We say that $X$ is \emph{$\Family$-acyclic} up to
dimension~$n$ if $\RHoml{k}(X)=0$ for every $k\leq n$.  Note
that being $\Family$-acyclic up to dimension~$-1$ is equivalent to
the condition that $X^{\Subgroup}\neq \emptyset$ for every
$\Subgroup\in \Family$.

The following is completely analogous to Brown's original
article~\cite{brown-87}:
Let $n\in \N$.  A $\Group$-CW-complex $X$
is said to be \emph{$\Family$-$n$-good} if the following two 
conditions hold:
\begin{enumerate}
    \item $X$ is $\Family$-acyclic up to dimension~$n-1$, and 
    
    \item
    for every $p$-cell $\sigma$ of $X$, $p\leq n$, $\Family \intersect
    \Group_\sigma\subset \Family$ and the stabiliser $\Group_\sigma$
    of $\sigma$ is of type \BredonFP[(\Family \intersect
    \Group_\sigma)]{n-p}.
\end{enumerate}
A filtration $(X_{\alpha})_{\alpha\in I}$ of a
$\Gamma$-CW-complex $X$ by $\Gamma$-invariant subcomplexes is said to be
of \emph{finite $n$-type} if the
$n$-skeleta $X_{\alpha}^{(n)}$ are cocompact for all $\alpha\in I$.

A directed system of Bredon modules $(M_{\alpha})_{\alpha\in I}$ is
said to be \emph{essentially trivial} if for every $\alpha\in I$ there
exists $\beta\geq \alpha$ such that the homomorphism $M_{\alpha}\to
M_{\beta}$ is trivial.

\medskip

\begin{mainthm}
    Let $\Group$ be a group and $\Family$ a family of subgroups of
    $\Group$.  Let $X$ be an $\Family$-$n$-good $\Group$-CW-complex
    and let $(X_{\alpha})_{\alpha\in I}$ be a filtration by
    $\Group$-invariant subcomplexes of finite $n$-type.

    Then $\Group$ is of type $\BredonFP{n}$ if and only if the
    directed system $(\RHoml{k}(X_{\alpha}))_{\alpha\in I}$ of reduced
    Bredon homology modules is essentially trivial for all $k<n$.
\end{mainthm}

%

The importance of a directed system being essentially trivial stems
from the following fact, which is the analogue of
\cite[Lemma~2.1]{brown-87}.

\begin{introbs}
    \label{obs:essentially_trivial}
    A directed system $(M_\alpha)_{\alpha \in I}$ of
    $\OrbitCat$-modules is essentially trivial if and only if
    \begin{equation*}
	\varinjlim_{\alpha} \prod_{\Subgroup \in
	\Family}\prod_{J_\Subgroup} M_\alpha(\Group/\Subgroup) = 0
    \end{equation*}
    for every family of cardinals $(J_\Subgroup)_{\Subgroup \in
    \Family}$.
\end{introbs}


\medskip

We would like to thank Kai-Uwe Bux for suggesting to work on this result and discussing intermediate versions with us. We also gratefully acknowledge support through the SFB 701 in Bielefeld (both authors) and the SFB 878 in Münster (second author).

\section{Basic definitions and results on Bredon modules}

\label{sec:Bredon-basics}

This section is to collect basic definitions and facts related to
Bredon modules for further reference.  Unless stated otherwise the
results can be found in~\cite[pp.~162--169]{luck-89}
or~\cite[p.~7--27]{mislin-03}.  By a Bredon module we mean either a
left or a right Bredon module unless the variance is explicitly
mentioned.

Since the category of Bredon modules is a functor category it follows
that limits and colimits of Bredon modules are calculated component
wise.  In particular a sequence of Bredon modules $0\to M' \to M \to
M'' \to 0$ is exact if and only if the corresponding sequence
\begin{equation*}
    0\to M'(\Group/\Subgroup) \to M(\Group/\Subgroup) \to 
    M''(\Group/\Subgroup) \to 0
\end{equation*}
of abelian groups is exact for every $\Group/\Subgroup\in \OrbitCat$. 

For subgroups $\AnotherSubgroup$ and $\Subgroup$ of $\Group$ we denote
by $[\Group/\AnotherSubgroup, \Group/\Subgroup]_{\Group}$ the set of
all $\Group$-maps $\Group/\AnotherSubgroup \to \Group/\Subgroup$.
For a fixed subgroup $\Subgroup$ of $\Group$ we denote by
$\Z[\PlaceA,\Group/\Subgroup]_{\Group}$ the right $\OrbitCat$-module
which sends $\Group/\AnotherSubgroup\in \OrbitCat$ to the free abelian
group $\Z[\Group/\AnotherSubgroup, \Group/\Subgroup]_{\Group}$ on the
basis $[\Group/\AnotherSubgroup, \Group/\Subgroup]_{\Group}$.
The left $\OrbitCat$-module $\Z[\Group/\Subgroup]_{\Group}$ is defined
analogously.
The \emph{free objects} in $\BredonRCat$ are now precisely the direct sums
of $\Z[\PlaceA, \Group/\Subgroup]_{\Group}$ with $\Subgroup\in
\Family$.  Likewise the \emph{free objects} in $\BredonLCat$ are the
direct sums of the Bredon modules $\Z[\Group/\Subgroup,
\PlaceA]_{\Group}$ with $\Subgroup\in \Family$. In either case a free Bredon
module is \emph{finitely generated} if the direct sum can be taken to be finite.
An arbitrary Bredon module is finitely generated if it is the surjective image
of a finitely generated free module.

For any two Bredon modules $M$ and $N$ of the same variance the set of
morphisms between them is denoted by $\mor_{\Family}(M,N)$.  A Bredon
module $P$ is \emph{projective} if the functor $\mor_{\Family}(P,
\PlaceA)$ is exact.  This is the case if and only if $P$ is a direct
sumand of a free Bredon module.

The categorical tensor product~\cite[p.~45]{schubert-70} gives rise to
a \emph{tensor product over $\Family$}.  It assigns to a right
$\OrbitCat$-module $N$ and left $\OrbitCat$-module $M$ an abelian
group $N\otimes_{\Family}M$.  The $\OrbitCat$-module $N$ is said to be
\emph{flat} if the functor $N\otimes_{\Family}\PlaceA$ is exact. Every
projective Bredon module is flat.

There exists also the \emph{tensor product over $\Z$}.  For two Bredon
modules $M$ and $N$ of the same variance it is defined to be the
Bredon module $M\otimes N$, which evaluated at any
$\Group/\Subgroup\in \OrbitCat$ is given by $(M\otimes
N)(\Group/\Subgroup) = M(\Group/\Subgroup) \otimes
N(\Group/\Subgroup)$.

\begin{lem}
    \label{lem:yoneda-iso}
    Let $N$ be a right $\OrbitCat$-module and $\Subgroup \in \Family$.  
    Then there exists an isomorphism
    \begin{equation*}
	N \otimes_\Family \Z[\Group/\Subgroup,\PlaceA]_\Group \isom
	N(\Group/\Subgroup)
    \end{equation*}
    which is natural in $N$.
\end{lem}

This statement follows from a Yoneda type argument, see for
example~\cite[p.~9]{mislin-03}.

\medskip

If $\Subgroup$ is a subgroup of $\Group$ such that $\Family \intersect
\Subgroup \defeq \{\AnotherSubgroup \intersect \Subgroup \mid
\AnotherSubgroup \in \Family\}\subset \Family$, then there exists a
functor
\begin{equation*}
    I_\Subgroup\colon \OrbitCat[\Subgroup][\Family\cap \Subgroup] \to
    \OrbitCat
\end{equation*}
which sends $\Subgroup/\AnotherSubgroup$ to $\Group/\AnotherSubgroup$
for every $\AnotherSubgroup\in \Family\cap\Subgroup$.  The induction
functor
\begin{equation*}
    \Ind{\Subgroup}{\Group} \colon \OrbitCat[\Subgroup][\Family\cap \Subgroup]
\to
    \OrbitCat
\end{equation*}
and the restriction functor
\begin{equation*}
    \Res{\Group}{\Subgroup} \colon 
    \OrbitCat
\to
\OrbitCat[\Subgroup][\Family\cap \Subgroup]
\end{equation*}
with respect to $I_\Subgroup$ are defined in~\cite[p.~166]{luck-89}.

\begin{lem}
    \label{lem:properties-induction}
    The functor $\Ind{\Subgroup}{\Group}$ preserves the properties of
    being finitely generated and being projective.  Furthermore, it is
    an exact functor and $\Ind{\Subgroup}{\Group} \BredonZ =
    \Z[\PlaceA,\Group/\Subgroup]_{\Group}$.
\end{lem}

The first statement is from~\cite[p.~169]{luck-89}.  The second
statement is Lemma~2.9 and Lemma~2.7 in~\cite[p.~268]{symonds-05}.

\begin{lem}
    \label{lem:properties-restriction}
    The functor $\Res{\Group}{\Subgroup}$ is exact and preserves
    being projective.
\end{lem}

For the first part of the this statement see~\cite[p.~169]{luck-89},
the remaining part is \cite[Lemma~3.7]{martinez-perez-02}.  We also
need the following special case of Proposition~3.5 in
\cite{martinez-perez-02}:

\begin{lem}
 \label{lem:MP02-Prop3.5}
  There exists a natural isomorphism
  \begin{equation*}
    (N\otimes \Z[\PlaceA,\Group/\Subgroup]_{\Group}) \otimes_\Family M
    \isom \Res{\Group}{\Subgroup} N \otimes_{\Family\cap\Subgroup}
    \Res{\Group}{\Subgroup} M
  \end{equation*}
  for any right $\OrbitCat$-module $N$ and any left $\OrbitCat$-module
  $M$.
\end{lem}

If $\Delta$ is a $\Group$-set, then we denote by $\Family(\Delta)$ the
set of stabilisers of $\Delta$.  We denote by
$\Z[\PlaceA,\Delta]_{\Group}$ the right $\OrbitCat$-module which sends
$\Group/\Subgroup\in \OrbitCat$ to the free abelian group on the basis
$[\Group/\Subgroup, \Delta]_{\Group}$ which is by definition the set of all
$\Group$-maps from $\Group/\Subgroup\to \Delta$.

\begin{lem}
  \label{lem:Q-tensor-Cp-flat}
  Let $\Delta$ be a $\Group$-set such that $\Family\cap
  \Subgroup\subset \Family$ for every $\Subgroup\in \Family(\Delta)$.
  For every projective right $\OrbitCat$-module $Q$ the
  $\OrbitCat$-module $Q\otimes \Z[\PlaceA, \Delta]_{\Group}$ is flat.
\end{lem}

\begin{proof}
  Since tensoring over $\Z$ is an additive functor, it is enough to
  verify the claim in the case that $\Delta = \Group/\Subgroup$ for
  some $\Subgroup$.  Since $Q$ is projective, it follows that the
  $\OrbitCat[\Subgroup][\Family\cap \Subgroup]$-module
  $\Res{\Group}{\Subgroup} Q$ is projective and hence flat.
  Furthermore $\Res{\Group}{\Subgroup}$ is an exact functor.
  Thus the functor, which sends any left $\OrbitCat$-module $M$ to
  $\Res{\Group}{\Subgroup} Q \otimes_{\Family\cap\Subgroup}
  \Res{\Group}{\Subgroup} M$ is exact.  Hence, in the light of the
  natural isomorphism of Lemma~\ref{lem:MP02-Prop3.5} it follows that
  tensoring $Q\otimes \Z[\PlaceA,\Group/\Subgroup]_{\Group}$ over
  $\Family$ is an exact functor, that is $Q\otimes
  \Z[\PlaceA,\Group/\Subgroup]_{\Group}$ is flat.
\end{proof}

For left $\OrbitCat$-module $M$ the left derived functors of
$\PlaceA\otimes_{\Family} M$ are denoted by
$\Tor_{*}^{\Family}(\PlaceA, M)$.  The following is a key ingredient
to our proof and can be found as Theorem~5.4 in
\cite{martinez-perez-11}:


\begin{prop}[Bieri--Eckmann Criterion for Bredon homology]
    \label{prop:Bredon-Bieri-Eckmann}
    Let $N$ be a right $\OrbitCat$-module and let $n\in \N$.  The
    following are equivalent:
    \begin{enumerate}
	\item $N$ is of type $\BredonFP{n}$.
	
	\item Let $(J_\Subgroup)_{\Subgroup \in \Family}$ be a family of
	cardinals.  The natural map
	\begin{equation*}
	    \Tor_k^\Family(N, \prod_{\Subgroup\in\Family}
	    \prod_{J_\Subgroup} \Z[\Group/\Subgroup,
	    \PlaceA]_{\Group}) \to \prod_{\Subgroup\in\Family}
	    \prod_{J_\Subgroup} \Tor_k^\Family(N, \Z[\Group/\Subgroup,
	    \PlaceA]_{\Group})
	\end{equation*}
	is an isomorphism for $k< n$ and an epimorphism for $k=n$.
    \end{enumerate}
\end{prop}

Note that $\Tor_{k}^{\Family}(N,\Z[\Group/\Subgroup,\PlaceA]_{\Group})
= 0$ for every $\Subgroup\in \Family$ and $k\geq 1$.  Thus the
requirement in~(ii) that the natural map is an epimorphism is
automatically satisfied for $k\geq 1$.

\medskip

The \emph{Bredon homology} $H_{*}^{\Family}(\Group; M)$ of $\Group$
with coefficients in the left $\OrbitCat$-module $M$ are defined to be
the groups $\Tor_{*}^{\Family}(\BredonZ, M)$.  Analogous to the
classical case~\cite[p.~172]{brown-82}, we
define the \emph{equivariant Bredon homology}
$H_{*}^{\Family}(X, M)$ of a $\Group$-CW-complex with coefficients
in the left $\OrbitCat$-module $M$ as follows, cf.\ \cite{dembegioti-11}.
Let $Q_{*}$ be a
projective resolution of the trivial $\OrbitCat$-module $\Z$ by right
$\OrbitCat$-modules.  Then we have the bigraded complex
\begin{equation*}
    (\BredonC{*}(X)\otimes Q_{**})\otimes_{\Family} M
\end{equation*}
of abelian groups.  We define $H_{*}^{\Family}(X, M)$ to be the
homology of the total complex of this bicomplex.  Note that
$H_{*}^{\Family}(\Group, M) = H_{*}^{\Family}(\pt, M)$.


\section{The case $n=0$}

In the classical case, being of type \FP{0} for a group is an empty
condition.  In the context of Bredon homology this is not true any
more.  Kochloukova, Martínez-Pérez and Nucinkis \cite[Lemma
2.3]{kochloukova-11} have given a characterisation of when a group is
of type \BredonFP{0}:

\begin{prop}
\label{prop:fp0}
    A group $\Group$ is of type \BredonFP{0} if and only if there
    is a finite subset $\Family_0$ of $\Family$ such that every
    $\Subgroup \in \Family$ is subconjugate to some element of
    $\Family_0$, i.e.~there is a $g \in \Group$ and a
    $\AnotherSubgroup \in \Family_0$ such that $\Subgroup^g \le
    \AnotherSubgroup$.
\end{prop}

Using this result, the case $n=0$ of the Main Theorem is readily
verified:

\begin{proof}[Proof of the Main Theorem for $n=0$]
    First assume that the directed system $(\RHoml{-1}(X_\alpha))_{\alpha \in I}$
    is essentially trivial.
    Then there is a $\beta \in I$ such that $X_\beta$ is
    $\Family$-acyclic up to dimension $-1$, i.e.\ $X^\Subgroup$ is
    non-empty for every $\Subgroup \in \Family$.  By assumption
    $\Skeleton{X_\beta}{0}$ is finite modulo $\Group$.  The stabilizer
    $\Group_x$ of every $x \in \Skeleton{X_\beta}{0}$ is of type
    \BredonFP[(\Family \intersect \Group_x)]{0}.  Hence there is a finite
    subset $\Family_{x,0}$ of $\Family \intersect \Group_x$ such that
    every $\Subgroup \in \Family \intersect \Group_x$ is subconjugate to
    some element of $\Family_{x,0}$.  Let $\Sigma_0$ be a set of
    representatives for $\Skeleton{X_\beta}{0}$ modulo $\Group$ an let
    \[
    \Family_0 = \Union_{x \in \Sigma_0} \Family_{x,0}
    \]
    which is a finite subset of $\Family$.  If $\Subgroup \in \Family$
    is arbitrary, then $X_\beta^\Subgroup$ is non-empty.  Hence
    $\Subgroup$ fixes some point $x$ of $\Skeleton{X_\beta}{0}$ and
    therefore is subconjugate to some element of $\Family_{x,0}$.

    Conversely assume that $\Group$ is of type \BredonFP{0}.  Let
    $\Family_0 \subseteq \Family$ be finite such that every element of
    $\Family$ is subconjugate to some element of $\Family_0$.  For
    arbitrary $\alpha \in I$ let $\beta \ge \alpha$ be such that
    $X_\beta$ contains a fixed point of each element of $\Family_0$.
    Let $\Subgroup \in \Family$ be arbitrary and $\Subgroup^g \le
    \AnotherSubgroup \in \Family_0$.  If $x \in X_\beta$ is a fixed
    point of $\AnotherSubgroup$, then $g.x \in X_\beta$ is a fixed
    point of $\Subgroup$.
\end{proof}


\section{Proof of the Main Theorem}

The following proposition is contained in \cite{dembegioti-11} for the case $n = \infty$ and the proof is essentially the same. We reproduce it for convenience.

\begin{prop}
    \label{prop:aux-1} 
    Let $X$ be a $\Group$-CW-complex which is $\Family$-acyclic up to
    dimension $n-1$ and let $M$ be a left $\OrbitCat$-module.  Then
    the natural isomorphism
    \begin{equation*}
        \BredonHoml{k}(X,M) \to \BredonHoml{k}(\Group,M)
    \end{equation*}
    (induced by the projection of $X$ to a point) is an isomorphism for $k < n$.
\end{prop}

\begin{proof}
Let $C_*$ be the chain complex of Bredon modules $X$ and let $Q_*$ be a projective resolution of $\BredonZ$. By definition there is a spectral sequence
\[
E^1_{pq} = H_q((C_* \otimes Q_p) \otimes_\Family M) \Rightarrow \BredonHoml{k}(X,M) \text{ .}
\]

We claim that
\begin{equation}
\label{eq:tensor_exact_sequence}
C_n \otimes Q_p \to C_{n-1} \otimes Q_p \to \ldots \to C_0 \otimes Q_p \to Q_p \to 0
\end{equation}
is a partial flat resolution of Bredon modules. By acyclicity of $X$ up to dimension $n-1$ the sequence
\[
C_n \to C_{n-1} \to \ldots \to C_0 \to \BredonZ \to 0
\]
is exact. To see that \eqref{eq:tensor_exact_sequence} is exact we have to see that it is exact evaluated at every orbit $\Group/\Subgroup$. This is true because $Q_p(\Group/\Subgroup)$ is free abelian. Flatness follows from Lemma~\ref{lem:Q-tensor-Cp-flat} because every $C_p$ is of the form $\Z[\PlaceA,\Delta]_\Group$.

It follows from \eqref{eq:tensor_exact_sequence} that $E^1_{pq} =
\Tor^\Family_q(Q_p,M)$ for $q < n$. Since $Q_p$ is projective we get
\[
\Tor^\Family_q(Q_p,M) =
\left\{
\begin{array}{ll}
Q_p \otimes_\Family M&q=0\\
0 &0 < q < n \text{ .}
\end{array}
\right.
\]

But $Q_* \otimes_\Family M$ can be used to compute $\BredonHoml{*}(\Group,M)$, therefore
\[
E^2_{pq} = \left\{
\begin{array}{ll}
\BredonHoml{p}(\Group,M)& q = 0\\
0 & 0 < q < n \text{ .}
\end{array}
\right.
\]
Since the triangle $p+q < n$ remains stable, this closes the proof.
\end{proof}

\begin{prop}
    \label{prop:aux-2} 
    Let $X$ be a $\Group$-CW-complex with cocompact $n$-skeleton.
    Assume that every $p$-cell $\sigma$ of $X$, $p\leq n$, the
    following two condition hold: $\Family \cap \Group_{\sigma}
    \subset \Family$, and $\Group_{\sigma}$ is of type $(\Family
    \intersect \Group)_\sigma\text{-}\FP{n-p}$.  Then for $k \le
    n$ and every family of cardinals $(J_\Subgroup)_{\Subgroup \in
    \Family}$ there exists an isomorphism
    \begin{equation*}
	\BredonHoml{k}(X,\prod_{\Subgroup \in
	\Family}\prod_{J_\Subgroup}\Z[\Group/\Subgroup,\PlaceA]_\Group)
	\to \prod_{\Subgroup \in
	\Family}\prod_{J_\Subgroup}\Homl{k}(X)(\Group/\Subgroup)
    \end{equation*}
    that is natural in $X$.
\end{prop}

\begin{proof}
    
    As in the previous proof let $C_{*} = \BredonC{*}(X)$ and let
    $Q_{*}\epi \BredonZ$ be a projective resolution of the trivial
    $\OrbitCat$-module.
    
    There exists a spectral sequence converging to 
    $\BredonHoml{*}(X,M)$ whose $E^{1}$-sheet is given by
    \begin{equation*}
        E^{1}_{pq} = H_{q}((C_{p}\otimes Q_{*})\otimes_{\Family} M).
    \end{equation*}
    Since $C_{p}(\Group/\Subgroup)$ is a free abelian group for every $\Subgroup\in 
    \Family$ it follows that $C_{p}\otimes Q_{*}$ is a resolution of 
    $C_{p}\otimes \Z = C_{p}$. Moreover this resolution is flat by 
    Lemma~\ref{lem:Q-tensor-Cp-flat} and thus there exist
    isomorphisms
    \begin{equation*}
        H_{q}((C_{p}\otimes Q_{*})\otimes_{\Family} M) \isom 
	\Tor^{\Family}_{q}(C_{p},M)
    \end{equation*}
    which are natural in $X$ and $M$.

    Next we show that $C_p$ is of type \FP{n-p} for $p \le n$.
    Note that the last statement
    of  Lemma~\ref{lem:properties-induction} implies
    \begin{equation*}
	C_p \isom \coprod_{\sigma \in \Sigma_p}
	\Ind{\Group_\sigma}{\Group} \BredonZ
    \end{equation*}
    where $\Sigma_p$ is a set of representatives for the $p$-cells of
    $X$ modulo $\Group$.  Note also, that $\Sigma_p$ is finite.  By
    assumption $\BredonZ$ is of type \FP{n-p} as an
    $\OrbitCat[\Group_{\sigma}][\Family \intersect
    \Group_\sigma]$-module for every $p$-cell $\sigma\in \Sigma_{p}$.
    The claim now follows from Lemma~\ref{lem:properties-induction}.

    Now take $M$ to be $\prod_{\Subgroup \in
    \Family}\prod_{J_\Subgroup}\Z[\Group/\Subgroup,\PlaceA]_\Group$
    and consider the spectral sequence above.  Since $C_p$ is of type
    \FP{n-p} the Bieri--Eckmann Criterion,
    Proposition~\ref{prop:Bredon-Bieri-Eckmann}, implies that
    \begin{equation*}
	E^1_{pq} = \prod_{\Subgroup \in \Family}\prod_{J_\Subgroup}
	\Tor_q^{\Family}(C_p,\Z[\Group/\Subgroup,\PlaceA]_\Group)
    \end{equation*}
    which is $0$ for $q > 0$.  The entry $E^1_{p0}$ is natural
    isomorphic to $\prod_{\Subgroup \in \Family}\prod_{J_\Subgroup}
    C_p \otimes_\Family \Z[\Group/\Subgroup,\PlaceA]_\Group$ and the
    differentials are induced by the differentials of the chain
    complex~$C_{*}$.  Therefore, one can read off the second page
    of the spectral sequence that
    \begin{equation*}
	\BredonHoml{k}(X,\prod_{\Subgroup \in
	\Family}\prod_{J_\Subgroup}\Z[\Group/\Subgroup,\PlaceA]_\Group)
	\isom \prod_{\Subgroup \in \Family}\prod_{J_\Subgroup} H_k(C_*
	\otimes_\Family \Z[\Group/\Subgroup,\PlaceA]_\Group) \isom
	\prod_{\Subgroup \in \Family}\prod_{J_\Subgroup}
	H_k(C_*(\Group/\Subgroup))
    \end{equation*}
    for $k < n$ where the last isomorphism is the isomorphism from
    Lemma~\ref{lem:yoneda-iso}.  But
    $H_k(C_*(\Group/\Subgroup))$ is just
    $\Homl{k}(X)(\Group/\Subgroup)$ and this concludes the proof.
\end{proof}

\begin{lem}
  \label{lem:aux-3} 
  Let $X$ be a $\Group$-CW-complex and let $(X_\alpha)_{\alpha\in I}$
  be a filtration of $X$ by $\Group$-invariant subcomplexes. Then the
  inclusions $X_\alpha \hookrightarrow X$ induce an isomorphism
  \begin{equation*}
    \varinjlim_\alpha \BredonHoml{*}(X_\alpha, M) \to \BredonHoml{*}(X,M)
  \end{equation*}
  for all left $\OrbitCat$-modules $M$.
\end{lem}

\begin{proof}
  This is due to the fact that $\varinjlim_{\alpha}$ is a
  filtered colimit and in particular exact.
\end{proof}

\begin{proof}[Proof of the Main Theorem]
    Since we have already covered the case $n=0$ we may and do assume
    that $n \ge 1$.  By the Bieri--Eckmann Criterion $\Group$ is of
    type \BredonFP{n} if and only if for every family
    $(J_\Subgroup)_{\Subgroup \in \Family}$ the natural map
    \begin{equation*}
	\varphi\: \BredonHoml{k}(\Group, \prod_{\Subgroup \in
	\Family}\prod_{J_\Subgroup}
	\Z[\Group/\Subgroup,\PlaceA]_\Group) \to \prod_{\Subgroup \in
	\Family} \prod_{J_\Subgroup} \BredonHoml{k}(\Group,
	\Z[\Group/\Subgroup,\PlaceA]_\Group)
	\tag{$*$}
	\label{eq:main}
    \end{equation*}
    is an isomorphism for $0 \le k < n$ and an epimorphism for $k =
    n$.  Since the right hand side is $0$ for $k > 0$ and since we are
    assuming that $n \ge 1$, the statement about the epimorphism is
    trivially satisfied.

    \smallskip

    Since the codomain of $\varphi$ is trivial for $k\geq
    1$ we first show that also the domain of $\varphi$
    is trivial for $0 < k < n$ (which is a special case of the
proof for $k=0$ below). We have the isomorphisms
    \begin{equation*}
	\BredonHoml{k}(\Group,\prod_{\Subgroup \in
	\Family}\prod_{J_\Subgroup}
	\Z[\Group/\Subgroup,\PlaceA]_\Group) \isom
	\BredonHoml{k}(X,\prod_{\Subgroup \in
	\Family}\prod_{J_\Subgroup}
	\Z[\Group/\Subgroup,\PlaceA]_\Group)
    \end{equation*}
    from Proposition~\ref{prop:aux-1}
    \begin{equation*}
	\BredonHoml{k}(X,\prod_{\Subgroup \in
	\Family}\prod_{J_\Subgroup}
	\Z[\Group/\Subgroup,\PlaceA]_\Group) \isom \varinjlim_{\alpha}
	\BredonHoml{k}(X_\alpha,\prod_{\Subgroup \in
	\Family}\prod_{J_\Subgroup}
	\Z[\Group/\Subgroup,\PlaceA]_\Group)
    \end{equation*}
    from Lemma~\ref{lem:aux-3} and
    \begin{equation*}
	\varinjlim_{\alpha} \BredonHoml{k}(X_\alpha,\prod_{\Subgroup
	\in \Family}\prod_{J_\Subgroup}
	\Z[\Group/\Subgroup,\PlaceA]_\Group) \isom \varinjlim_{\alpha}
	\prod_{\Subgroup \in \Family}\prod_{J_\Subgroup}
	\Homl{k}(X_\alpha)(\Group/\Subgroup)
    \end{equation*}
    from
    Proposition~\ref{prop:aux-2}.
    By Observation~\ref{obs:essentially_trivial} this is trivial if
    and only if the system $(\Homl{k}(X_\alpha))_{\alpha \in I}$ of
    $\OrbitCat$-modules is essentially trivial.

    \smallskip

    \pagebreak[3]

    For \nopagebreak the remaining case $k=0$ consider the following commuting
    diagram (where we dropped the index sets for readability):
    \begin{diagram}
	\BredonHoml{0}(\Group,\prod
	\Z[\Group/\Subgroup,\PlaceA]_\Group) & \rTo^{\varphi} &
	\prod
	\BredonHoml{0}(\Group,\Z[\Group/\Subgroup,\PlaceA]_\Group)
        \\
	\uTo & & \uTo
        \\
	\BredonHoml{0}(X,\prod \Z[\Group/\Subgroup,\PlaceA]_\Group) &
	\rTo & \prod
	\BredonHoml{0}(\text{pt.},\Z[\Group/\Subgroup,\PlaceA]_\Group)
        \\
	\uTo & & \uTo
        \\
	\varinjlim_\alpha \BredonHoml{0}(X_\alpha,\prod
	\Z[\Group/\Subgroup,\PlaceA]_\Group) & \rTo &
        \varinjlim_\alpha \prod \BredonHoml{0}(\text{pt.},
        \Z[\Group/\Subgroup,\PlaceA]_\Group)
        \\
	\dTo & & \dTo
        \\
	\varinjlim_\alpha \prod \Homl{0}(X_\alpha)(\Group/\Subgroup) &
	\rTo^\psi & \varinjlim_\alpha \prod
	\Homl{0}(\text{pt.})(\Group/\Subgroup)
    \end{diagram}
    The vertical arrows of the top square are isomorphisms by
    Proposition~\ref{prop:aux-1}. The vertical arrows of the middle
    square are induced by the inclusions $X_\alpha \hookrightarrow X$
    and the indentity on the one point space respectively; it follows from
    Lemma~\ref{lem:aux-3} that they are
    isomorphisms. Finally, the vertical arrows of the bottom square
    are the isomorphisms from Proposition~\ref{prop:aux-2}.

    Since all the vertical arrows in the diagram are
    isomorphisms it follows that $\varphi$ is an
    isomorphism if and only if $\psi$ is an
    isomorphisms. But $\psi$ fits into the
    short exact sequence
    \begin{equation*}
	0 \to \varinjlim_\alpha \prod_{\Subgroup \in
	\Family}\prod_{J_\Subgroup}
	\RHoml{0}(X_\alpha)(\Group/\Subgroup) \to \varinjlim_\alpha
	\prod_{\Subgroup \in \Family}\prod_{J_\Subgroup}
	\Homl{0}(X_\alpha)(\Group/\Subgroup)
	\stackrel{\psi}{\to} \varinjlim_\alpha \prod_{\Subgroup
	\in \Family}\prod_{J_\Subgroup} \Z \to 0
    \end{equation*}
    and it follows Observation~\ref{obs:essentially_trivial}
    that $\psi$ (and therefore $\varphi$) is an
    isomorphism if and only if the system
    $(\RHoml{0}(X_\alpha))_{\alpha \in I}$ of $\OrbitCat$-modules is
    essentially trivial.
\end{proof}



\providecommand{\bysame}{\leavevmode\hbox to3em{\hrulefill}\thinspace}
\providecommand{\MR}{\relax\ifhmode\unskip\space\fi MR }
\providecommand{\MRhref}[2]{%
  \href{http://www.ams.org/mathscinet-getitem?mr=#1}{#2}
}
\providecommand{\href}[2]{#2}

\end{document}